\documentclass[reqno]{amsart}

\usepackage[abbrev]{amsrefs}		
\usepackage{hyperref}

\usepackage{multirow}
\usepackage{amsmath,amsfonts}
\usepackage{amssymb,amsthm}
\usepackage{mathrsfs,stmaryrd}
\usepackage[all,cmtip]{xy}
\usepackage{tensor}
\usepackage[makeroom]{cancel}
\usepackage{enumitem}

\theoremstyle{plain}
\newtheorem{defi}{Definition}[section]
\newtheorem{lemma}[defi]{Lemma}
\newtheorem{prop}[defi]{Proposition}
\newtheorem{thm}[defi]{Theorem}
\newtheorem{cor}[defi]{Corollary}

\newtheorem{notation}[defi]{Notation}

\theoremstyle{remark}
\newtheorem{rmk}{Remark}

\newtheoremstyle{mystyle}
  {}
  {}
  {\itshape}
  {}
  {\bfseries}
  {.}
  { }
  {}
\theoremstyle{mystyle}

\newcounter{note}
\DeclareMathAlphabet{\mathpzc}{OT1}{pzc}{m}{it}

\def\shO{\mathcal{O}}

\def\idI{\mathscr{I}}

\def\Proj{\mathbb{P}}
\def\cxC{\mathbb{C}}

\def\mapto{\longrightarrow}

\def\isom{\cong}

\def\codim{{\rm codim}\,}

\def\Ric{{\rm Ric\thinspace}}

\def\d{{\text{d}}}

\def\p{\partial}
\def\wg{\wedge}
\def\cxi{\sqrt{-1}}
\def\<{\langle}
\def\>{\rangle}

\def\dbar{\bar{\partial}}

\def\curv{\Theta}
\def\tens{\otimes}
\def\etens{\boxtimes}

\newcommand{\bigwg}[1]{\bigwedge\nolimits^{\! #1}}
\newcommand{\bigtens}[1]{\bigotimes\nolimits^{\! #1}}

\title[Syzygies of A Tower of Manifolds]{Syzygies of A Tower of Compact Local Hermitian Symmetric Spaces of Finite Type}

\author{Yih Sung}
\email{yih.sung@usu.edu}
\address{Dept.~of Mathematics and Statistics, Utah State University, Logan, UT 84322}

\begin{document}

\renewcommand{\theenumi}{\alph{enumi}}

\maketitle

\begin{abstract}
Let $X$ be a $n$ dimensional compact local Hermitian symmetric space of non-compact type and $L=\shO(K_X)\tens\shO(qM)$ be an adjoint line bundle. Let $c>0$ be a constant. Assume the curvature of $M$ is $\ge c\omega$, where $\omega$ is the k\"ahler form of $X$, and $X$'s injectivity radius has a lower bound $\tau>\sqrt{2e}$, where $e$ is the Euler number. In this article, we prove that if $q>\frac{2e}{c\tau} \cdot (p+1)n$, then $L$ enjoys Property $N_p$. Applying this result to a tower of compact local Hermitian symmetric spaces $\cdots\mapto X_{s+1}\mapto X_s\mapto\cdots\mapto X_0=X$, we prove that $2K_{s}$ has Properties $N_p$ for $s\gg 0$ and fixed $p$. Based on the same technique, we show a criterion of projective normality of algebraic curves and a division theorem with small power difference.
\end{abstract}


\section{Introduction}
\subsection{Background And Main Results}
Let $X$ be a compact local Hermitian symmetric space of non-compact type, namely $X$ can be written as $G/H$, where $G$ is a semi-simple Lie group of non-compact type and $H$ is a maximal compact subgroup. Under this natural setting, we consider a tower of manifolds $X_s=\Gamma_s\backslash G/H$ such that $\Gamma_{s+1}<\Gamma_s$ is a normal subgroup, which associates a sequence of finite maps
$$\cdots\mapto X_{s+1}\mapto X_s\mapto\cdots\mapto X_0=X.$$
In \cite{ref_yeung20} and \cite{ref_yeung_note1}, S.-K. Yeung shows that for $s\gg 0$ the canonical bundle $K_s$ is very ample and can separate the $k$-th jet. Inspired by the result of very ampleness, we investigate the properties of higher normality, namely Property $N_p$. In particular, Property $N_0$ corresponds to projective normality. We expect that for enough high tower covering, the higher normality should also be satisfied. In particular, if $X=B^n/\Gamma$ is a ball quotient with the injectivity radius $\rho_X$ and $L$ is numerically equivalent to $qK_X$, in \cite{ref_hwang&to13} J.-M. Hwang and W.-K. To show that $K_X\tens L$ enjoys Property $N_p$ if $p$ and $q$ satisfy
$$
\frac{p+1}{q}<\frac{2(n+1)}{n}\cdot\sinh^2\Big( \frac{\rho_X}{4} \Big).
$$
For fixed $q\ge 2$ and $p$, in the tower of $X$, $qK_{s}$ has Property $N_p$ for $s\gg 0$. In this article we want to generalize this result to local symmetric Hermitian spaces of non-compact type and show that $qK_s$ has Property $N_p$ for fixed $q\ge 2$, $p$ and $s\gg 0$. Our main theorem is
\begin{thm}\label{thm_main}
Let $X$ be a compact k\"ahler manifold and $\omega$ be the k\"ahler metric. Let the curvature form $R_M$ of the holomorphic line bundle $M$ satisfy $R_M\ge c\omega, c>0$. Fix an integer $q\ge 1$ and let $\tau$ be the injectivity radius of the manifold $X$. Then, if
\begin{enumerate}
\item $\tau$ is bounded from below: $\tau> \sqrt{2e}$, where $e$ is the Euler number and
$$
q\ge \frac{1}{c}\cdot\sqrt{\frac{2}{e}} (p+1)n >
\frac{2e}{c\tau} \cdot (p+1)n,
$$

\item or
$$
q> \frac{2e}{c}\Big( \frac{1}{\tau}+\frac{1}{\tau^2} \Big) \cdot (p+1)n,
$$
\end{enumerate}
$K_X+qM$ satisfies Property $N_p$.
\end{thm}

If the tower of covering is enough high, the injectivity radius will approach infinity. Thus, we have the following direct implication.
\begin{cor}\label{cor_main}
Let $X$ be a compact local Hermitian symmetric space of non-compact type, and $L=qK_X$, $q\ge 2$. Let $\{X_s\}$ be a tower of covering of $X$. Then, for fixed $p$, there exists $s_0$ such that $L_s=qK_{X_s}$ satisfies Property $N_p$ for every $s\ge s_0$.
\end{cor}
Suggested by cohomological criterion of Property $N_0$, it has the format of division theorem. Let $L=K_X+M$. Property $N_0$ is equivalent to the surjectivity of the map
$$
\beta_k:H^0(X,L)\tens H^0(X,L^{\tens k})\mapto H^0(X,L^{\tens (k+1)})
$$
for every $k\ge 1$. If we intend to apply Skoda's division theorem \cite{ref_sko72} to $\beta_k$, $k$ has to be large. In this article, we remove this constrain by introducing the injectivity radius into the estimate.
\begin{cor}[Division Theorems with small power difference]\label{cor_div_thm_wi_smll_gaps}
Let $P_\tau\subset\cxC^n$ be a polydisc with side length $\tau$ with respect to a K\"ahler metric $\omega$. Let the curvature form $R_M$ of the holomorphic line bundle $M$ satisfy $R_M\ge c\omega, c\ge 0$. If
\begin{enumerate}
\item $\tau$ is bounded from below: $\tau> \sqrt{2e}$, where $e$ is the Euler number

\item and
$$
c>
\frac{2e}{k\tau}\cdot n,
$$
\end{enumerate}
the map $\beta_k:H^0(P_\tau,L)\tens H^0(P_\tau,L^{\tens k})\mapto H^0(P_\tau,L^{\tens (k+1)})$ is surjective.

\end{cor}

By using similar techniques, we can prove a theorem of projective normality of algebraic curves.
\begin{thm}\label{thm_N_0_property_of_curves}
Let $X$ be a Riemann surface and $\omega$ be a k\"ahler metric. Let the curvature form $R_M$ of the holomorphic line bundle $M$ satisfy $R_M\ge c\omega, c\ge 0$. Let $\tau$ be the injectivity radius of the manifold $X$. If $\tau> \sqrt{2e}$ and $c> \frac{2e}{\tau}$, then $K_X+M$ satisfies Property $N_0$.
\end{thm}

\begin{rmk}
For conveniency, we will introduce two small positive constants $\epsilon$ and $\epsilon'$ in doing later estimates, so the conditions in Theorem \ref{thm_main} can be phrased in terms of $\epsilon$ and $\epsilon'$:
\begin{enumerate}
\item $\tau$ is bounded from below: $\tau\ge \sqrt{e(2+\epsilon')}$, where $e$ is the Euler number and
$$
q\ge \frac{1}{c}\cdot\sqrt{\frac{2+\epsilon'}{e}} (p+1)(n+\epsilon) \ge
\frac{e(2+\epsilon')}{c\tau} \cdot ((p+1)n+\epsilon),
$$

\item or
$$
q\ge
\frac{e(2+\epsilon')}{c}\Big( \frac{1}{\tau}+\frac{1}{\tau^2} \Big)
\cdot ((p+1)n+\epsilon),
$$
\end{enumerate}
and so are the conditions in Corollary \ref{cor_div_thm_wi_smll_gaps} and Theorem \ref{thm_N_0_property_of_curves}.
\end{rmk}


\subsection{Contents}
This article is structured as follows: in Section \ref{sec_N_p_property} we review the definition and equivalent cohomological characterization of Property $N_p$. In Section \ref{sec_pf_of_main_thm}, we use the extension theorem to prove the main Theorem \ref{thm_main}. In Section \ref{sec_pf_of_cor_and_thm_on_curve} we use the techniques developed in Section \ref{sec_pf_of_main_thm} to show the projective normality of algebraic curves (Theorem \ref{thm_N_0_property_of_curves}) and division theorem with small power difference (Corollary \ref{cor_div_thm_wi_smll_gaps}).


\section*{Acknowledgements}
We want to specially thank professor Sai-Kee Yeung for useful discussion and generous advice on this paper.


\section{$N_p$ Properties}\label{sec_N_p_property}
\subsection{Setting}
Let $X$ be an irreducible projective variety, and $L$ be a very ample line bundle on $X$ defining an embedding
$$
\phi_L:X\mapto \Proj=\Proj H^0(X,L).
$$
Consider the graded ring $R_L=R(X,L)=\bigoplus H^0(X,L^{\tens m})$ determined by $L$, and write $S={\rm Sym}\, H^0(X,L)$ for the homogeneous coordinate ring of $\Proj$. Then $R_L$ admits a free resolution $E_\bullet$:
$$
\xymatrix{
\cdots \ar[r] & \bigoplus_j S(-a_{2,j}) \ar[r] & 
\bigoplus_j S(-a_{1,j}) \ar[r] &
S \bigoplus \big(\bigoplus_j S(-a_{0,j})\big) \ar[r] &
R_L \ar[r] & 0
}.
$$
We hope the resolution for the first $p$ terms in $E_\bullet$ are as simple as possible. For example, every two adjacent grades are just different by $1$, i.e. 
\begin{defi}[Property $N_p$, \cite{ref_lazarsfeld04}, Definition 1.8.50]
The embedding line bundle $L$ satisfies Properties $N_p$ if $E_0=S$, and
$$a_{i,j}=i+1\;\text{ for all $j$}$$
whenever $1\le i\le p$.
\end{defi}

There is an alternative way to characterize Property $N_p$ by cohomologies. Let $X$ be a projective variety and $L$ be a line bundle generated by global sections. Then, there exists a natural exact sequence:
\begin{equation}\label{ex_seq_of_lin_bdl_gen_by_glob_sec}
\xymatrix{
0 \ar[r] & M_L \ar[r] & H^0(L)\tens\shO_X \ar[r]^-\iota & L \ar[r] & 0
},
\end{equation}
where $M_L$ is the kernel of $\iota$, and (\ref{ex_seq_of_lin_bdl_gen_by_glob_sec}) naturally induces a Koszul complex:
\begin{equation}
\xymatrix{
0 \ar[r] & \bigwedge^{p+1} M_L \ar[r] & \bigwedge^{p+1} H^0(L)\tens\shO_X
\ar[r] & \bigwedge^p M_L\tens L \ar[r] & 0
},
\end{equation}
where the map is described by
$$
f_{i_1,\cdots,i_{p+1}}e^{i_1}\wg\cdots\wg e^{i_p} \mapsto 
\sum (-1)^k f_{i_1,\cdots,i_{p+1}}s_{i_k} 
e^{i_1}\wg\cdots\widehat{e^{i_k}}\cdots\wg e^{i_{p+1}}.
$$


\subsection{Cohomological Criterion of Property $N_p$}
Under the preceding setting, the criterion of Property $N_p$ is as follows.
\begin{lemma}[\cite{ref_EL93}, Lemma 1.6]\label{lem_cohom_N_p}
Assume that $L$ is very ample, and that $H^1(X,L^k)=0$ for all $k\ge 1$. Then $L$ satisfies Property $N_p$ iff
$$
H^1(X,\bigwg{a}M_L\tens L^b)=0,\; 
\text{ $\forall\, a\le p+1$ and $b\ge 1$. }
$$
\end{lemma}
In characteristic zero case, the wedge product is a direct summand of the tensor product. Therefore, $L$ will have Property $N_p$ if the following condition holds:
$$
H^1(X,\bigtens{a} M_L\tens L^b)=0,\; 
\text{ $\forall\, a\le p+1$ and $b\ge 1$. }
$$
In general, it is hard to deal with $M_L$ directly, so an improved version of vanishing condition is needed. The idea is to consider vanishing of cohomology groups on the product of $X$ rather than $X$ itself.
\begin{prop}[\cite{ref_inamodar97}, Lemma 1.5]\label{prop_vanish}
Let $L$ be an ample line bundle on a projective manifold $X$ with $H^1(X,L^{\tens k})=0$ for all $k\ge 1$. Then for an integral $\ell\ge 2$, $L$ satisfies Property $N_{\ell-2}$ if for all integers $m$ and $b$ satisfying $2\le m\le \ell$ and $b\ge 1$,
$$
H^1(X^m,q_1^* L^{\tens b}\tens q_2^* L
\tens\cdots\tens q_m^*L\tens \idI_\Sigma)=0,
$$
where $\Sigma=D_{1,m}\cup D_{2,m}\cup\cdots\cup D_{m-1,m}$ is the union of pairwise diagonals in $X\times\cdots\times X$.
\end{prop}
This proposition has a direct implication.
\begin{thm}[\cite{ref_inamodar97}, Theorem 1.6]
Let $X$ be a projective variety and let $L$ be an ample line bundle on $X$. Then for every positive integer $p_0$, there exists a number $n_0$ such that $L^n$ has property $N_{p_0}$ for every $n\ge n_0$.
\end{thm}
Remark that in this theorem, $L$ has to be raised to enough high power without upper bound estimate, but in our Theorem \ref{thm_main}, the required power of the line bundle is effective and extplicit.

\section{Proof of The Main Theorem}\label{sec_pf_of_main_thm}
In this section we will break the proof of Theorem \ref{thm_main} into several steps. Basically, we will follow the framework proposed in \cite{ref_inamodar97}. Nervertheless, instead of applying Kodaira Vanishing theorem, we will use extension theorems in oder to obtain the effective power of $L=qK_X$. Throughout this section, we will assume
\begin{equation}\label{eq_assump_of_vanishing}
H^1(X,L^b)=0
\end{equation}
for all $b\ge 1$. Remark that the assumption of Theorem \ref{thm_main} satisfies this condition since $R_M>0$.

\subsection{$p=0$ case} According to the argument in \cite{ref_inamodar97}, it is sufficient to justify the conditions in the following lemma to show Property $N_0$.
\begin{lemma}[Lemma 1.1, \cite{ref_inamodar97}] \label{lem_ind_base_case}
Denote $\idI_D$ the ideal sheaf of the diagonal embedding of $X$ in $X^2=X\times X$. Then
\begin{enumerate}
\item $H^0(X,M_L\tens L^{\tens(b+1)})=H^0(X^2,\idI_D\tens L\tens L^{\tens(b+1)}$),

\item $H^1(X^2,\idI_D\tens L\tens L^{\tens(b+1)})=0\Rightarrow H^1(X,M_L\tens L^{\tens(b+1)})=0$,
\end{enumerate}
for every $b\ge 1$.
\end{lemma}
The main ingredient of the the proof in \cite{ref_inamodar97} is Kodaira-Viehweg vanishing theorem. Instead of applying the vanishing theorem, we prove the vanishing by extending the sections in the cohomology groups. Let
\begin{equation}
V=H^0(X,L).
\end{equation}
Consider exact sequence (\ref{ex_seq_of_lin_bdl_gen_by_glob_sec}) and tensor it with $L^{\tens(b+1)}$, which induces a long exact sequence
\begin{equation}\label{ex_seq_of_lin_bdl_gen_by_glob_sec_ind}
\xymatrix{
0 \ar[r] & H^0(M_L\tens L^{\tens (b+1)}) \ar[r] & 
V\tens H^0(L^{\tens(b+1)}) \ar[r]^-\iota & H^0(L^{\tens(b+2)}) & \\
\ar[r] & H^1(M_L\tens L^{\tens(b+1)}) \ar[r] & H^1(L^{\tens(b+1)})=0. & &
}
\end{equation}
The last term is vanishing because of the assumption (\ref{eq_assump_of_vanishing}). Thus, if we can show the map $\iota$ is surjective, which implies $H^1(X,M_L\tens L^{\tens(b+1)})=0$. Then the original arguments in \cite{ref_inamodar97} follow and we are done.

\subsubsection{Setting for Applying Extension Theorem}\label{subsubsec_setting_for_applying_ext_thm}
Before we procced the extension theorem, let us treat the exact sequence (\ref{ex_seq_of_lin_bdl_gen_by_glob_sec_ind}) as the cohomology groups of line bundle on $X\times X$. Consider the exact sequence on $X\times X$: 
$$
\xymatrix{
0 \ar[r] & 
L\tens L^{b+1}\tens \idI_D \ar[r]& 
L\tens L^{b+1} \ar[r]^-{res} & L^{b+2}\tens\shO_D
\ar[r] & 0,
}
$$
where $D\mapto X\times X$ is the diagonal embedding of $X$. Then, there are natural isomorphisms:
\begin{gather*}
V\tens H^0(X,L^{\tens(b+1)}) \isom 
H^0(X\times X,\pi_1^*L\tens \pi_2^*L^{\tens (b+1)}) \\
H^0(X,L^{\tens(b+2)})\isom H^0(D,\pi_1^*L\tens \pi_2^*L^{\tens (b+1)}).
\end{gather*}
\begin{notation}
Let $L,M$ be line bundles on $X$. We will denote $L\etens M$ the line bundle $\pi_1^*L\tens \pi_2^*M$ on $X\times X$ for short later.
\end{notation}
Hence the extension problem of the map $\iota$ becomes a problem of extending sections of $L\etens L^{\tens(b+1)}$ on the diagonal $D\subset X\times X$. Let $\dim X=n$. Since the codimension of $D\subset$ in $X\times X$ is greater than $1$ if $\dim X\ge 2$, we need to blow up $D$ on $X\times X$ to fix this issue. Let $\alpha  :Y=Bl_D X\times X\mapto X\times X$ be the blowup, and we have the following diagram
$$
\xymatrixcolsep{0.3pc}\xymatrix{
E \ar[d] & \subset & Y\ar[d]^-\alpha \\
D & \subset & X\times X
}
$$
where $E$ is the exceptional divisor. Then, we turn to consider the extension problem on $Y$:
$$
\xymatrix{
H^0(Y,\alpha^* L\etens L^{\tens(b+1)}) \ar[r]\ar[d]^\parallel &
H^0(E,\alpha^* L\etens L^{\tens(b+1)})\ar[d]^\parallel \\
H^0(X\times X, L\etens L^{\tens(b+1)})\ar[r] &
H^0(D, L\etens L^{\tens(b+1)})
}.
$$
By the Ohsawa-Takegoshi theorem (\cite{ref_ohsawa04} theorem 1.1), we need to justify the curvature condition:
\begin{align*}
\cxi\curv(\alpha^* L\etens L^{\tens(b+1)}) + \Ric_{Y} \ge 
(1+\epsilon)\cxi \curv(E) \\
\Longleftrightarrow \cxi\curv(\alpha^* L\etens L^{\tens(b+1)}) \ge
\cxi\curv(\alpha^* K_X\etens K_X)+(n+\epsilon)\cxi\curv(E)
\end{align*}
where $0<\epsilon\ll 1$. Here we utilize the blowup formula 
\begin{equation*}
\Ric_Y=-K_Y=-(\alpha^*K_{X\times X}+({\rm codim}\,D-1)E)
\end{equation*}
and $\codim D=2n-n=n$. In particular, if $L=(q+1)K_X$ and $q\ge 1$, the above curvature condition becomes
$$
\cxi\curv(\alpha^*K_X\etens K_X)+
\cxi\curv(\alpha^*(p_1^* K_X^{\tens (q-1)}+p_2^* 
K_X^{\tens ((q+1)b+(q-1))}) \ge 
(n+\epsilon)\cxi\curv(E).
$$
Since $\cxi\curv(K_X)> 0$, we only need to require
\begin{equation}\label{curv_cond_for_N_0}
q\cxi\curv(\alpha^*K_X\etens K_X)\ge (n+\epsilon)\cxi\curv(E).
\end{equation}
Note that $X\times X$ is also a Hermitian symmetric space of non-compact type because $\cxi\curv(K_{X\times X})>0$ and bounded below by using the Hermitian-Einstein metric. Let us examine the bundle $\shO_Y(E)$ closely. Choose an appropriate hypersurface $H$ on $X\times X$ such that $\shO_Y(E)$ is trivial on $U=X\times X-H$, namely, the transition function $h_{\alpha\beta}$ on the intersection of open sets $U_\alpha\cap U_\beta$ is $1$. Let $U_D=(E-H)|_{D}\subset (D-H)$ is a local open set. After blowing up, the exceptional divisor $E\mapto D$ is a $\Proj^{n-1}$ projective bundle and 
$$E|_{U_D}\cong D\times \Proj^{n-1}.$$
Let $x=(x_1,x_2)\in D\subset X\times X$, take $\Omega=B_\tau(x_1)\times B_\tau(x_2)\subset X\times X$, so that in $\Omega$, $X\times X$ can be seen as flat. Take an open set $U'\subset U$ if necessary so that $E|_{U'}\cong D\times \cxC^{n-1}\subset Y|_{U'}$. Let $\Omega'=\Omega\cap U'$ Then, on $\Omega'$, we are able to choose local coordinates
$$z_1,\cdots,z_n$$
for $D|_{\Omega'}$,
$$w_1,\cdots,w_{n-1}$$ 
for the exceptional direction, and extend the set to
$$
w_1,\cdots,w_{n-1},z_1,\cdots,z_n,z_{n+1},
$$
which is the coordinate system of $Y|_{U'}$. In particular we can choose $z_{n+1}$ so that
$$
E= \{z_{n+1}=0\}.
$$


\subsubsection{Metric of $\shO_Y(E)$}\label{subsec_curv_of_shO_Y(E)}
By the standard technique in proving the Kodaira vanishing theorem, we take a two-open sets covering to cover $Y|_{U'}$. Denote
$$
P_\epsilon=\{x=(w,z)\in Y|_{U'} ; |z_{n+1}|<\epsilon \},
$$
and then consider 
$$
V_1=P_{2\epsilon}\, \text{ and } \,V_2=Y|_{U'}-P_\epsilon.
$$
Fix $0<\epsilon'\ll 1$ a small constant, and we need a technical lemma to construct a special cut-off function.
\begin{lemma} 
There exists a cut-off function $\chi$ such that
\begin{gather}
\notag\chi(t)=1,t\le \frac{\tau}{2},\;\; \chi(t)=0,t\ge\tau \\
\label{est_of_chi}-\frac{2+\epsilon'}{\tau^2}\le \chi'(t)\le 0, \;\;
|\chi''(t)|\le \frac{4(2+\epsilon')}{\tau^2}.
\end{gather}
\end{lemma}
\begin{proof}
This technical cut-off function is constructible. We refer the details to the proof of Theorem 1 in \cite{ref_yeung20}. 
\end{proof}
Let us consider the partition of unity functions associated with $\chi$:
\begin{gather*}
\rho_1(z)=\rho(z)=e^{-\chi(\sigma)}\, 
\text{ and } \,\rho_2(z)=1-\rho(z), \\
\text{where } \sigma=|z_{n+1}|^2 \Longrightarrow 0\le \rho_1,\rho_2\le 1.
\end{gather*}
By construction, it is easy to see
\begin{equation}
\frac{1}{e}\le e^{-\chi}\le 1\, \text{ and }\, 
|\chi'(\sigma)|\le \frac{(2+\epsilon')}{\tau^2},
\end{equation}
where $e$ is the Euler number. 

Now we are ready to construct a metric on $\shO_Y(E)$. Let
\begin{equation}\label{metric_h_1}
\begin{cases}
h_1=e^{-\varphi}=1+|w|^2:=1+|w_1|^2+\cdots+|w_{n-1}|^2\ge 1 \; \text{ and } \\
h_2=1.
\end{cases}
\end{equation}
Note that $h_1$ is a natural metric of $\shO(-1)$. Then, we define
\begin{equation}\label{metric_of_shO(E)}
\begin{split}
h&=\rho_1 h_1+\rho_2 h_2=\rho h_1+(1-\rho)h_2 \\
&=\rho (h_1-1)+1=\rho(e^{-\varphi}-1)+1.
\end{split}
\end{equation}
This metric is well defined. Since on $U'\subset U$, $\shO(E)|_{U'}$ is trivial. Thus, the transition function is $1$, which allows us to manipulate the metric freely without worrying about the transition laws. Note that the curvature induced by $h$ has signs:
$$
\cxi\curv(\shO_{Y,x}(E))=
\begin{cases}
0 & \text{on $Y_x-P_{2\epsilon}$}, \\
\text{bounded} & \text{on $P_{2\epsilon}-P_\epsilon$}, \\
\le 0 & \text{on $P_{\epsilon}$ ($=0$ alog radius direction)}, \\
<0 & \text{on $E$}.
\end{cases}
$$


\subsubsection{Curvature of $\shO_Y(E)$}
Let us further explore the curvature of $\shO_Y(E)$ with respect to the metric $h$ defined in (\ref{metric_of_shO(E)}). We have
\begin{align*}
&\p_k\p_{\bar\ell}\log h \\
&=\frac{(\rho(e^{-\varphi}-1)+1) \Big(\p_k\p_{\bar\ell}\rho \,(e^{-\varphi}-1)+
\p_{\bar\ell}\rho \,\p_k e^{-\varphi}+\p_k\rho \p_{\bar\ell}e^{-\varphi}+
\rho\p_k\p_{\bar\ell}e^{-\varphi} \Big)}
{(\rho(e^{-\varphi}-1)+1)^2} \\
&\makebox[12pt]{}
-\frac{(\p_{\bar\ell}\rho(e^{-\varphi}-1)+\rho\p_{\bar\ell}e^{-\varphi})
(\p_k\rho(e^{-\varphi}-1)+\rho\p_k e^{-\varphi})}
{(\rho(e^{-\varphi}-1)+1)^2} \\
&=\frac{1}{(\rho(e^{-\varphi}-1)+1)^2} \Big(
\underbrace{(\rho(e^{-\varphi}-1)+1) 
(\p_k\p_{\bar\ell}\rho (e^{-\varphi}-1)+
\rho\p_k\p_{\bar\ell}e^{-\varphi})}_{(a)} \\
&\makebox[94pt]{}
+\underbrace{(\rho(e^{-\varphi}-1)+1)
(\p_{\bar\ell}\rho\, \p_k e^{-\varphi}+\p_k\rho \p_{\bar\ell}e^{-\varphi})
}_{(b)} \\
&\makebox[94pt]{}
-\underbrace{(\p_{\bar\ell}\rho(e^{-\varphi}-1)
+\rho\p_{\bar\ell}e^{-\varphi})
(\p_k\rho(e^{-\varphi}-1)+\rho\p_k e^{-\varphi})}_{(c)}
\Big),
\end{align*}
where $(b)-(c)$ is:
\begin{align*}
\p_{\bar\ell}\rho \,\p_k e^{-\varphi}+\p_k\rho\p_{\bar\ell}e^{-\varphi}
-\p_{\bar\ell}\rho \,\p_k\rho(e^{-\varphi}-1)^2-\rho^2\p_{\bar\ell}e^{-\varphi}
\p_k e^{-\varphi}.
\end{align*}
Therefore, $\p_k\p_{\bar\ell}\log h$ is
\begin{equation}\label{eq_ddbar_h}
\begin{split}
&\frac{1}{\rho(e^{-\varphi}-1)+1}
(\p_k\p_{\bar\ell}\rho (e^{-\varphi}-1)+
\rho\p_k\p_{\bar\ell}e^{-\varphi}) \\
&+\frac{1}{(\rho(e^{-\varphi}-1)+1)^2}
(\p_{\bar\ell}\rho \,\p_k e^{-\varphi}+\p_k\rho \,\p_{\bar\ell}e^{-\varphi}
-\p_{\bar\ell}\rho \,\p_k\rho(e^{-\varphi}-1)^2-\rho^2\p_{\bar\ell}e^{-\varphi}
\p_k e^{-\varphi}).
\end{split}
\end{equation}
Recalling that
$$
-\p_k\p_{\bar\ell}\log\rho=-\frac{\p_k\p_{\bar\ell}\rho}{\rho}
+\frac{\p_{\bar\ell}\rho\, \p_k\rho}{\rho^2},
$$
we aim to identify such shapes in (\ref{eq_ddbar_h}). Matching up the terms and introducing a tangent vector $v$,  we have
\begin{equation}\label{ddbar_log_h_v}
\begin{split}
-\p_k\p_{\bar\ell}\log h\, v^k\bar v^\ell
=&-\frac{\p_k\p_{\bar\ell}\rho}{\rho}v^k\bar v^\ell A
+\frac{\p_k\rho\,\p_{\bar\ell}\rho}{\rho^2}v^k\bar v^\ell A^2 \\
&-\frac{\p_k\p_{\bar\ell}e^{-\varphi}}{e^{-\varphi}}v^k\bar v^\ell B
+\frac{\p_ke^{-\varphi}\,\p_{\bar\ell}e^{-\varphi}}
{(e^{-\varphi})^2}v^k\bar v^\ell B^2 \\
&-\frac{1}{(\rho(e^{-\varphi}-1)+1)^2}
(\p_k\rho \p_{\bar\ell}e^{-\varphi}+\p_{\bar\ell}\rho\p_k e^{-\varphi})
v^k\bar v^\ell,
\end{split}
\end{equation}
where
$$
A=\frac{\rho (e^{-\varphi}-1)}{\rho(e^{-\varphi}-1)+1},\;
B=\frac{\rho e^{-\varphi}}{\rho(e^{-\varphi}-1)+1}.
$$
By (\ref{metric_h_1}), it is easy to see
$$
A\le 1,\;\text{ and }\;
B=\frac{\rho e^{-\varphi}}{\rho e^{-\varphi}+(1-\rho)}
\le 1.
$$
Let us further investigate the term $\frac{\p_k\rho\,\p_{\bar\ell}\rho}{\rho^2}v^k\bar v^\ell$ in (\ref{ddbar_log_h_v}). Compute
$$
\frac{\p_k\rho\,\p_{\bar\ell}\rho}{\rho^2}v^k\bar v^\ell
=\p_k\chi v^k\cdot \p_{\bar\ell}\chi \bar v^\ell=|\p_k\chi v^k|^2\ge 0,
$$
which implies
\begin{align}
\notag -\p_k\p_{\bar\ell}\log h\, v^k\bar v^\ell
\le&\Big( -\frac{\p_k\p_{\bar\ell}\rho}{\rho}
+\frac{\p_k\rho\,\p_{\bar\ell}\rho}{\rho^2} \Big)v^k\bar v^\ell A \\
\notag &+\Big( -\frac{\p_k\p_{\bar\ell}e^{-\varphi}}{e^{-\varphi}}
+\frac{\p_ke^{-\varphi}\,\p_{\bar\ell}e^{-\varphi}}
{(e^{-\varphi})^2} \Big)v^k\bar v^\ell B \\
\notag &-\frac{1}{(\rho(e^{-\varphi}-1)+1)^2}
(\p_k\rho \p_{\bar\ell}e^{-\varphi}+\p_{\bar\ell}\rho\p_k e^{-\varphi})
v^k\bar v^\ell \\
\label{ddbar_log_h_v_ineq} \le& \p_k\p_{\bar\ell}\chi\, v^k\bar v^\ell+
\p_k\p_{\bar\ell}\varphi\, v^k\bar v^\ell+
\frac{2}{(\rho e^{-\varphi})^2}
|\p_k\rho \p_{\bar\ell}e^{-\varphi}v^k \bar v^\ell|.
\end{align}
\begin{itemize}
\item Regarding the first term in (\ref{ddbar_log_h_v_ineq}), we will take care of it by multiplying $e^{-\chi}$ to the metric of $M$. (cf. (\ref{eq_metric_of_M}))

\item Regarding the second term in (\ref{ddbar_log_h_v_ineq}), recall $\rho=e^{-\chi(\sigma)}$ and compute
\begin{equation}\label{pos_part_of_ddbar_log_h_v}
\begin{split}
\frac{2}{(\rho e^{-\varphi})^2}
|\p_k\rho\, \p_{\bar\ell}e^{-\varphi}v^k \bar v^\ell|
&=\frac{2}{\rho e^{-\varphi}}
|\p_k\chi v^k \, \p_{\bar\ell}\varphi \bar v^\ell| 
=\frac{|\chi'|}{\rho e^{-\varphi}}\cdot
2|\p_k\sigma v^k, \p_{\bar\ell}\varphi \bar v^\ell| \\
&\le \frac{e |\chi'|}{e^{-\varphi}}
(\underbrace{|\p_k\sigma v^k|^2}_{(d)}+ 
\underbrace{|\p_{\bar\ell}\varphi \bar v^\ell|^2)}_{(e)},
\end{split}
\end{equation}
where $\sigma=|z_{n+1}|^2$. Note that here we estimate $1/\rho=e^{\chi}\le e$ because $0\le \chi\le 1$.

\end{itemize}


\subsubsection{Estimates of (d) and (e) in (\ref{pos_part_of_ddbar_log_h_v})}\label{subsec_est_of_(d)&(e)}
Regarding term (e), we aim to combine the estimate of term $(e)=\frac{e|\chi'|}{e^{-\varphi}}|\p_{\bar\ell}\varphi \bar v^\ell|^2$ and the negativity of $\p_k\p_{\bar\ell}\varphi\, v^k\bar v^\ell$ in (\ref{ddbar_log_h_v_ineq}). By using the explicit expression $\varphi=-\log(1+|w|^2)$ and $|w|^2=\sum_k|w_k|^2$, we can compare these two terms. Recall 
$$
-\p_k\p_{\bar\ell}\log (1+|w|^2)=
-\frac{(1+|w|^2)\delta_{k\ell}-w_\ell\bar w_k}{(1+|w|^2)^2},
$$
which implies
\begin{align*}
\p_k\p_{\bar\ell}\varphi v^k\bar v^\ell&=
\frac{-(1+|w|^2)|v|^2+w_\ell\bar w_k v^k\bar v^\ell}{(1+|w|^2)^2} \\
&=\frac{-|v|^2}{(1+|w|^2)^2}-\frac{|w|^2|v|^2}{(1+|w|^2)^2}
+\frac{|w_\ell\bar v^\ell|^2}{(1+|w|^2)^2} \\
&\le \frac{-|v|^2}{(1+|w|^2)^2}.
\end{align*}
On the other hand,
\begin{align*}
\frac{1}{e^{-\varphi}}|\p_{\bar\ell}\varphi \bar v^\ell|^2
=\frac{1}{1+|w|^2} \Big|\frac{w_\ell \bar v^\ell}{1+|w|^2}\Big|^2
\le \frac{1}{(1+|w|^2)^3}|w|^2 |v|^2
\le \frac{|v|^2}{(1+|w|^2)^2}\le |v|^2. 
\end{align*}
Thus, by using (\ref{est_of_chi}), we require the numerical condition 
$$
e|\chi'|\le e\cdot \frac{2+\epsilon'}{\tau^2}\le 1
\Longleftrightarrow \tau \ge \sqrt{e(2+\epsilon')}.
$$

If this is not the case, namely, the injectivity radius is small, then we need the positivity of $M$ to take over the positivity of $|\p_{\bar\ell}\varphi \bar v^\ell|^2$. Again, by choosing the normal coordinates plus $\sqrt{-1}\curv(M)\ge c\omega$, we have $\sqrt{-1}\curv(M)_{k\bar\ell}v^k\bar v^\ell\ge c|v|^2$.
Thus, we require the numerical condition
$$
qc \ge (n+\epsilon)e\frac{(2+\epsilon')}{\tau^2},
$$
which corresponds to the $1/\tau^2$ term in condition (b) in Theorem \ref{thm_main}.

Let us examine term (d). By introducing the estimate of term (e), we obtain the estimate:
\begin{equation}\label{final_est}
\begin{split}
-\p_k\p_{\bar\ell}\log h\, v^k\bar v^\ell
&\le \p_k\p_{\bar\ell}\chi\, v^k\bar v^\ell+
\frac{e |\chi'|}{e^{-\varphi}} |\p_k\sigma v^k|^2 \\
&= \p_k\p_{\bar\ell}\chi\, v^k\bar v^\ell
+\frac{e}{1+|w|^2} |\chi'| \cdot |(\p_k z_{n+1}) v^k|^2 \\
&\le \p_k\p_{\bar\ell}\chi\, v^k\bar v^\ell
+e\frac{(2+\epsilon')}{\tau^2}\cdot\tau |v|^2.
\end{split}
\end{equation}
Recall (\ref{curv_cond_for_N_0}), we require
\begin{align*}
q\cxi\curv(\alpha^*M\etens M)\ge (n+\epsilon)\cxi\curv(E),
\end{align*}
namely, we need
$$
qc \ge (n+\epsilon)\cdot e\frac{(2+\epsilon')}{\tau},
$$
which is the condition (a) in Theorem \ref{thm_main} in the case of $p=0$.


\subsubsection{Extension Theorem}
Let $e^{-\varphi_M}$ be the smooth metric of $M$ such that $\p_k\p_{\bar\ell}\varphi_M v^k\bar v^\ell \ge c |v|^2$, and equip $M$ with the metric 
\begin{equation}\label{eq_metric_of_M}
e^{-(\varphi_M+\frac{\chi}{q})}.
\end{equation}
Let $L=K_X+qM$. By the construction,
$$
\p_k\p_{\bar\ell}(q\varphi_M+\chi)v^k\bar v^\ell
\ge (n+\epsilon)e\frac{(2+\epsilon')}{\tau} |v|^2
+\p_k\p_{\bar\ell}\chi v^k\bar v^\ell
\ge -\p_k\p_{\bar\ell}\log h\, v^k\bar v^\ell.
$$
Then, a section $f\in H^0(D,(K_X\tens qM)\etens(K_X\tens qM)^{\tens(b+1)})$ on the diagonal $D$ satisfying the $L^2$ condition
$$
\int_{D}\|f\|^2 \d V_{D}=
\int_{U} |f|^2 e^{-q(b+2)(\varphi_M+\frac{\chi}{q})}\, \d V_{U} <\infty.
$$
can be extended to $\tilde{f}$ on $\Omega\times\Omega$ with $L^2$ estimates, and then $\tilde{f}$ can be extended to $X\times X$ and be a section in $H^0(X\times X,L\etens L^{\tens(b+1)})$. Since the metric $e^{-(\varphi_M+\chi)}$ is smooth, every section in $H^0(D,L\etens L^{\tens(b+1)})$ is extendible. Thus, the map $\iota$ is surjective as desired and $K_X+qM$ has Property $N_0$. In particular, if $M=K_X, q=1$, for enough high tower $X_s$, the injectivity radius will be sufficient large. Hence, $2K_{X_s}$ will enjoy Property $N_0$ for $s\gg 0$.


\subsection{$p=1$ case}
By Lemma \ref{lem_cohom_N_p}, Property $N_1$ is equivalent to $H^1(X,\tens^a M_L\tens L^{\tens b})=0$ for $2\ge a\ge 0$, $b\ge 1$ which implies the original cohomological condition $H^1(X,\bigwedge^a M_L\tens L^{\tens b})=0$ for $2\ge a\ge 0$, $b\ge 1$.

\subsubsection{Setting for Applying Extension Theorems}
We consider the exact sequence
$$
\xymatrix{
0 \ar[r] & M_L^{\tens 2}\tens L^{\tens(b+1)} \ar[r] & 
V\tens M_L\tens L^{\tens(b+1)} \ar[r]^-\iota & 
M_L \tens L^{\tens(b+2)} \ar[r] & 0
},
$$
and its induced long exact sequence
$$
\xymatrix{
0 \ar[r] & H^0(M_L^{\tens 2}\tens L^{\tens (b+1)}) \ar[r] & 
V\tens H^0(M_L\tens L^{\tens(b+1)}) \ar[r]^-\iota & 
H^0(M_L\tens L^{\tens(b+2)}) & \\
\ar[r] & H^1(M_L^{\tens 2}\tens L^{\tens(b+1)}) \ar[r] & H^1(M_L\tens L^{\tens(b+1)})=0. & &
}
$$
Note that the last vanishing $H^1(M_L\tens L^{\tens(b+1)})=0$ is by the previous step, i.e. Lemma \ref{lem_ind_base_case} (b). Our aim is to show $H^1(M_L^{\tens 2}\tens L^{\tens(b+1)})=0$. Similar to the proof of Lemma \ref{lem_ind_base_case}, it is sufficient to show the following lemma.
\begin{lemma}[Lemma 1.3, \cite{ref_inamodar97}] \label{lem_ind_p=2_case}
Assume that $H^1(M_L\tens L^{\tens(b+1)})=0$. Let $\Sigma^{(3)}=D_{1,3}\cup D_{2,3}$, where $D_{1,i}$ is the the diagonal embedding of $X$ in $X_1\times X_i\subset X\times X\times X$. Denote $\idI_{\Sigma^{(3)}}$ the ideal sheaf of $\Sigma^{(3)}$. Then
\begin{enumerate}
\item $H^0(X,M_L^{\tens 2}\tens L^{\tens(b+1)})=H^0(X^3,\idI_{\Sigma^{(3)}}\tens L\etens L\etens L^{\tens(b+1)}$),

\item $H^1(X^3,\idI_{\Sigma^{(3)}}\tens L\etens L\etens L^{\tens(b+1)})=0\Rightarrow H^1(X,M_L^{\tens 2}\tens L^{\tens(b+1)})=0$,
\end{enumerate}
for every $b\ge 1$.
\end{lemma}
The proof is similar to the proof of lemma \ref{lem_ind_base_case}. Instead of showing $H^1(X^3,\idI_{\Sigma^{(3)}}\tens L\etens L\etens L^{\tens(b+1)})=0$ directly, we proceed the proof by using extension theorem. Recall $V=H^0(X,L)$ and the canonical isomorphisms:
\begin{gather*}
\begin{align*}
V\tens H^0(X,M_L\tens L^{\tens(b+1)})
&\isom V\tens H^0(X^2,\idI_D\tens L\etens L^{\tens (b+1)}) \\
&\isom H^0(X^3,\idI_{D_{2,3}}\tens L\etens L\etens L^{\tens(b+1)})
\end{align*}
\\
\begin{align*}
H^0(X,M_L\tens L^{\tens(b+2)})
&\isom H^0(X^2,\idI_{D}\tens L\etens L^{\tens(b+2)}) \\
&\isom H^0(D_{1,3},\idI_{D_{2,3}}\tens L\etens L^{\tens (b+1)}),
\end{align*}
\end{gather*}
where $D_{i,j}\isom X^2\subset X\times X\times X$ is the partial diagonal embedding defined by $D_{i,j}=\{(x_1,x_2,x_3)\in X^3\mid x_i=x_j\}$. Then, the vanishing of $H^1(X^3,\idI_{\Sigma^{(3)}}\tens L\etens L\etens L^{\tens(b+1)})$ is equivalent to the surjectivity of the restriction map
\begin{equation}\label{ext_problem_of_p=1}
H^0(X^3,\idI_{D_{2,3}}\tens L\etens L\etens L^{\tens(b+1)})\mapto
H^0(D_{1,3},\idI_{D_{2,3}}\tens L\etens L^{\tens (b+2)}).
\end{equation}

\subsubsection{Extension Theorem}
In order to apply the Ohsawa-Takegoshi theorem (\cite{ref_ohsawa04} Theorem 1.1), we increase the dimension of the extension center by blowup. Let $\alpha :Y=Bl_{D_{1,3}} X^3\mapto X^3$ be the blowup, and we have the following diagram:
$$
\xymatrixcolsep{0.3pc}\xymatrix{
E \ar[d] & \subset & Y\ar[d]^-\alpha \\
D_{1,3} & \subset & X^3
}
$$
where $E$ is the exceptional divisor. Then, we turn to consider the extension problem on $Y$:
$$
\xymatrix{
H^0(Y,\alpha^* \idI_{D_{2,3}}\tens L\etens L\etens L^{\tens(b+1)}) 
\ar[r]\ar[d]^\parallel &
H^0(E,\alpha^* \idI_{D_{2,3}}\tens L\etens L^{\tens (b+2)})
\ar[d]^\parallel \\
H^0(X^3,\idI_{D_{2,3}}\tens L\etens L\etens L^{\tens(b+1)}) \ar[r] &
H^0(D_{1,3},\idI_{D_{2,3}}\tens L\etens L^{\tens (b+2)})
}.
$$
Note that the multiplier ideal sheaf $\idI_{D_{2,3}}$ associated to a singular weight function which takes $\infty$ along $D_{2,3}$. By the same justification as the $p=0$ case, we calculate the curvature conditions of the bundles $L\etens L\etens L^{\tens(b+1)}$, and require
$$
qc \ge (2n+\epsilon)\cdot e\frac{(2+\epsilon')}{\tau},
$$
to obtain the desired curvature estimate:
\begin{equation*}
q\cxi\curv(\alpha^*K_{X_s}\etens K_{X_s} \etens K_{X_s})
\ge (n+\epsilon)\cxi\curv(E).
\end{equation*}
The coefficient $2$ of $n$ is coming from the blow up formula 
\begin{equation*}
\Ric_Y=-K_Y=-(\alpha^*K_{X\times X\times X}+({\rm codim}\,D_{1,3}-1)E)
\end{equation*}
and $\codim D_{1,3}=3n-n=2n$. Thus, the morphism in (\ref{ext_problem_of_p=1}) is surjective and $K_X+qM$ has Property $N_1$. In particular, if $M=K_X, q=1$, for enough high tower $X_s$, the injectivity radius will be sufficient large. Hence, $2K_{X_s}$ will enjoy Property $N_1$ for $s\gg 0$.

\subsection{General case of $p$} 
Following the previous arguments in the $p=0$ and $p=1$ cases, we proceed the mathematical induction on $p$. Similar to Lemma \ref{lem_ind_base_case} and Lemma \ref{lem_ind_p=2_case}, we assume the following statements.

\smallskip

\noindent\textbf{Inductive Hypothesis.} 
Let the partial diagonal embedding of $X$ be
$$D_{i,j}=\{(x_1,\cdots,x_n)\in X^n \mid x_i=x_j\},$$
and denote $\Sigma^{(p)}=D_{1,2}\cup D_{2,3}\cup\cdots\cup D_{p-1,p}$. Then
\begin{enumerate}
\item $H^0(X,M_L^{\tens (p-1)}\tens L^{\tens(b+1)})=H^0(X^{p},\idI_{\Sigma^{(p)}}\tens L\etens \cdots \etens L\etens L^{\tens(b+1)}$),

\item $H^1(X^{p},\idI_{\Sigma^{(p)}}\tens L\etens \cdots \etens L\etens L^{\tens(b+1)})=0\Rightarrow H^1(X,M_L^{\tens (p-1)}\tens L^{\tens(b+1)})=0$,
\end{enumerate}
for every $b\ge 1$. We take the singular weight function of $\idI_{\Sigma^{(p)}}$ as follows. Locally, $\idI_{\Sigma^{(p)}}$ is choosen as
$$e^{-\varphi}=e^{-\varphi_{1,p}}\cdots e^{-\varphi_{p-1,p}},$$
where $e^{-\varphi_{i,p}}=\frac{1}{|f_{i,p}|^2}$ where $f_{i,p}$ is the local defining equation of $D_{i,p}$. By the construction, $\varphi$ is a plurisubharmonic function. Then, we can prove the general lemma by using extension theorem.

\begin{lemma}[Lemma 1.5, \cite{ref_inamodar97}]
Assume that $L$ has Property $N_{p-1}$. Denote $\idI_{\Sigma^{(p)}}$ the ideal sheaf of $\Sigma^{(p)}$. Then
\begin{enumerate}
\item $H^0(X,M_L^{\tens p}\tens L^{\tens(b+1)})=H^0(X^{p+1},\idI_{\Sigma^{(p+1)}}\tens L\etens \cdots \etens L\etens L^{\tens(b+1)}$).

\item $H^1(X^{p+1},\idI_{\Sigma^{(p+1)}}\tens L\etens \cdots \etens L\etens L^{\tens(b+1)})=0\Rightarrow H^1(X,M_L^{\tens p}\tens L^{\tens(b+1)})=0$.
\end{enumerate}
\end{lemma}
Recall $V=H^0(X,L)$ and the natural isomorphisms:
\begin{gather*}
\begin{align*}
V\tens H^0(X,M_L^{\tens (p-1)}\tens L^{\tens(b+1)})
&\isom V\tens H^0(X^p,\idI_{\Sigma^{(p)}}
\tens L\etens\cdots\etens L^{\tens (b+1)}) \\
&\isom H^0(X^{p+1},\idI_{\Sigma^{(p)}}
\tens L\etens\cdots\etens L\etens L^{\tens(b+1)})
\end{align*}
\\
\begin{align*}
H^0(X,M_L^{\tens(p-1)}\tens L^{\tens(b+2)})
&\isom H^0(X^p,\idI_{\Sigma^{(p)}}
\tens L\etens\cdots\etens L^{\tens (b+2)}) \\
&\isom H^0(D_{1,p},\idI_{\Sigma^{(p)}}\tens 
L\etens\cdots\etens L^{\tens (b+1)}),
\end{align*}
\end{gather*}
it is sufficient to solve the extension problem:
\begin{equation}\label{ext_problem_of_general_p}
H^0(X^{p+1},\idI_{\Sigma^{(p)}}
\tens L\etens\cdots\etens L\etens L^{\tens(b+1)}) \mapto
H^0(D_{1,p},\idI_{\Sigma^{(p)}}\tens 
L\etens\cdots\etens L^{\tens (b+1)}).
\end{equation}
Let be the $\alpha :Y=Bl_{D_{1,p}} X^{p+1}\mapto X^{p+1}$ blowup and $E$ be the exceptional divisor. Then, we apply the same argument as before. We require
$$
qc \ge ((p+1)n+\epsilon)\cdot e\frac{(2+\epsilon')}{\tau}
$$
to obtain the desired curvature estimate:
\begin{equation*}
\cxi\curv(\alpha^*K_{X_s}\etens \cdots \etens K_{X_s})
\ge (n+\epsilon)\cxi\curv(E)
\end{equation*}
Thus, the morphism in (\ref{ext_problem_of_general_p}) is surjective as desired and $K_X+qM$ has Property $N_p$. In particular, if $M=K_X, q=1$, for enough high tower $X_s$, the injectivity radius will be sufficient large. Hence, $2K_{X_s}$ will enjoy Property $N_p$ for $s\gg 0$ and we complete the proof.


\section{Proof of Corollary \ref{cor_div_thm_wi_smll_gaps} and Theorem \ref{thm_N_0_property_of_curves}} \label{sec_pf_of_cor_and_thm_on_curve}
\subsection{Normality of Riemann Surfaces}\label{sec_normal_of_riemann_surf}
In this section, we will use the same framework and techniques to handle projective normality of Riemann surfaces. It is well known that Property $N_0$ is equivalent to projective normality, namely, if $L$ is projective normal if
$$
Sym^k H^0(X,L)\mapto H^0(X,L^{\tens k})
$$
is surjective (cf. \cite{ref_green&lazarsfeld86}, Introduction). Particularly, if one can show that
$$
\beta_k: H^0(X,L)\tens H^0(X,L^{\tens (k-1)})\mapto H^0(X,L^{\tens k})
$$
is surjective for every $k$, then $L$ is projective normal. In our case, $L$ is an adjoint bundle, i.e. $L=K_X+M$. For sufficient large $k$, one can show that $\beta_k$ is surjective by Skoda's division theorem. Thus, the difficulty lies in the case when $k$ is small, especially, when $k=2$. The key is the natural isomorphisms we utilized in Section \ref{subsubsec_setting_for_applying_ext_thm}:
\begin{gather*}
V\tens H^0(X,L^{\tens(k-1)}) \isom 
H^0(X\times X,\pi_1^*L\tens \pi_2^*L^{\tens (k-1)}) \\
H^0(X,L^{\tens k})\isom H^0(D,\pi_1^*L\tens \pi_2^*L^{\tens (k-1)}),
\end{gather*}
and the blowup diagram:
$$
\xymatrix{
H^0(Y,\alpha^* L\etens L^{\tens(k-1)}) \ar[r]^-{res}\ar[d]^\parallel &
H^0(E,\alpha^* L\etens L^{\tens(k-1)})\ar[d]^\parallel \\
H^0(X\times X, L\etens L^{\tens(k-1)})\ar[r]^-{res} &
H^0(D, L\etens L^{\tens(k-1)})
},
$$
where $\alpha:Y=Bl_D X\times X$ is the blow-up along the diagonal. Then, we intend to apply the technique of extension similar to Section \ref{subsec_curv_of_shO_Y(E)} to show that the restriction map
$$
\beta_k:H^0(Y,\alpha^* L\etens L^{\tens(k-1)})
\mapto H^0(E,\alpha^* L\etens L^{\tens(k-1)})
$$
is surjective.

Recall the curvature estimate of $\shO_Y(E)$ (\ref{ddbar_log_h_v_ineq}) and (\ref{pos_part_of_ddbar_log_h_v}):
\begin{equation}
\begin{split}
-\p_k\p_{\bar\ell}\log h\, v^k\bar v^\ell &\le 
\p_k\p_{\bar\ell}\chi\, v^k\bar v^\ell+
\p_k\p_{\bar\ell}\varphi\, v^k\bar v^\ell+
\frac{2}{(\rho e^{-\varphi})^2}
|\p_k\rho \p_{\bar\ell}e^{-\varphi}v^k \bar v^\ell| \\
&\le \p_k\p_{\bar\ell}\chi\, v^k\bar v^\ell+
\p_k\p_{\bar\ell}\varphi\, v^k\bar v^\ell+
\frac{e |\chi'|}{e^{-\varphi}}
(\underbrace{|\p_k\sigma v^k|^2}_{(d)}+ 
\underbrace{|\p_{\bar\ell}\varphi \bar v^\ell|^2)}_{(e)}.
\end{split}
\end{equation}
When the injectivity radius $\tau$ is large, such as
$$\tau \ge \sqrt{e(2+\epsilon')},$$
we can apply the estimate in Section \ref{subsec_est_of_(d)&(e)} and use the negativity of $\shO(-1)$ to control the term (e). For term (d), we aim to find a global coordinate to construct a metric to take over it. Since $X$ is a Riemann surface, the universal cover $\tilde{X}$ is $\Proj^1,\cxC$ or a disc $B(0,1)$. By removing the branch points and the brach cuts, the fundamental domain $\Omega\subset \tilde X$ is biholomorphic to an open set $U\subset X$. By Riemann's theorem, $U$ is further biholomorphic to a disc. Thus, we have a global coordinate $z$ which enables us to construct an appropriate plurisubharmonic function with compact support to dominate term (d). Recall the estimate of the last term in (\ref{final_est}):
\begin{align*}
\frac{e |\chi'|}{e^{-\varphi}}|\p_k\sigma v^k|^2 \le
\frac{e(2+\epsilon')}{\tau^2}\cdot \tau |v|^2
=\frac{e(2+\epsilon')}{\tau}|v|^2.
\end{align*}
We aim to construct a plurisubharmonic function $\eta_\tau$ such that
$$
\p_z\p_{\bar z} \eta_\tau \ge \frac{e(2+\epsilon')}{\tau}.
$$
Consider
$$
\eta_\tau = 
\begin{cases}
c_1\tau |\frac{z}{\tau}|^2 & z<\tau \\
c_1\tau & z = \tau
\end{cases}
$$
where $c_1$ is a constant to be determined. Then,
\begin{align*}
\p_z\p_{\bar z}\eta_\tau &= c_1\tau \cdot \frac{1}{\tau^2}
\p_w\p_{\bar w} |w|^2
\; \text{ (where $w=\frac{z}{\tau}$)}\\
&= \frac{c_1}{\tau},\; \text{ on $|w|<1$}.
\end{align*}
Thus, we can take
\begin{equation}
f_\tau=e^{-\eta_\tau},\; \text{ and $c_1=e(2+\epsilon')$}
\end{equation}
so that $f_\tau$ is globally defined on $X$ because the extended values on the branch points and branch cuts are $1$ by taking the limit, and the curvature of the weight function $f_\tau$ is
$$
-\p\dbar \log f_\tau=\p\dbar \eta_\tau \ge 
\frac{e(2+\epsilon')}{\tau}.
$$
Therefore, we equip $L$ with the metric
$$e^{-(\varphi_M+\frac{\chi}{q}+\eta_\tau)},$$ 
which has the desired curvature estimate
\begin{align*}
q\cxi\curv(\alpha^*M\etens M)\ge (1+\epsilon)\cxi\curv(E).
\end{align*}
Hence, every section of $H^0(E,\alpha^* L\etens L^{\tens(k-1)})$ is extendible, and we finish the proof of Theorem \ref{thm_N_0_property_of_curves}.
\begin{rmk}
In applying the extension theorem, we require a smooth weight function so that every section on the extension center is $L^2$ finite. Here we skip a standard technical detail. In order to make the weight function smooth, we need a family of smoothifiers to smooth out the conner of $f_\tau$ at $z=\tau,0$, i.e. $w=1,0$. By taking limit, we can still obtain the desired estimates and extend the section.
\end{rmk}


\subsection{Division Theorem with Small Power Difference}
The key step in proving Theorem \ref{thm_N_0_property_of_curves} is to find a global coordinate. The setting of Corollary \ref{cor_div_thm_wi_smll_gaps}, the division theorem with small power difference, assures the existence of such coordinate, so by introducing the function $f_\tau$ constructed in Section \ref{sec_normal_of_riemann_surf} and the estimates in Section \ref{subsec_curv_of_shO_Y(E)} and Section \ref{subsec_est_of_(d)&(e)}, Corollary \ref{cor_div_thm_wi_smll_gaps} follows.




\end{document}